\documentclass[11pt]{article}
\usepackage[dvips]{graphicx}
\usepackage{amssymb,color}
\usepackage{epic}
\usepackage{pstricks}
\usepackage{pst-node}
\usepackage{pstricks-add}

\def\BBox{\kern  -0.2cm\hbox{\vrule width 0.2cm height 0.2cm}}

\newtheorem{lemma}{Lemma}[section]

\newtheorem{theorem}{Theorem}[section]
\newtheorem{definition}{Definition}[section]
\newtheorem{corollary}{Corollary}[section]

\newtheorem{remark}{Remark}[section]
\newtheorem{conjecture}[theorem]{Conjecture}
\newtheorem{construction}[theorem]{Construction}

\title{A Family of Dense Mixed Graphs of Diameter $2$}
\author{G. Araujo-Pardo$^{1}$, C. Balbuena$^{2}$,
M. Miller$^{3,4}$, M. \v{Z}d\'imalov\'a$^5$
\thanks{ \footnotesize{\em Email addresses:} ~ garaujo@matem.unam.mx (G. Araujo),
~ m.camino.balbuena@upc.edu (C. Balbuena), \, \, ~ Mirka.Miller@newcastle.edu.au (M. Miller),\, \, ~ zdimalova@math.sk (M. \v{Z}d\'imalov\'a)}
 \\[2ex]
\\$^1${\footnotesize Instituto de Matem\'{a}ticas, Universidad Nacional Aut\'{o}noma de M\'{e}xico,} \\
{\footnotesize M\'{e}xico D. F., M\'exico }\\
$^2${\footnotesize Departament de Matem\`atica Aplicada III, Universitat
Polit\`ecnica de Catalunya, }\\
{\footnotesize Campus Nord, Edifici C2, C/ Jordi Girona 1 i 3 E-08034 Barcelona,
Spain.} \\
$^3${\footnotesize School of Mathematical and Physical Sciences, University of Newcastle, Australia.}\\
$^4${\footnotesize Department of Mathematics, University of West Bohemia, Pilsen, Czech Republic.}\\
$^5${\footnotesize Department of Mathematics and Descriptive Geometry,}\\
{\footnotesize Slovak University of Technology in Bratislava, Slovakia. }
}

\begin{document}
\maketitle

\begin{abstract}
A mixed graph is said to be {\it{dense}} if its order is close to the Moore bound and it is  {\it{optimal}} if there is not a mixed graph with the same parameters and bigger order.

We present a construction that provides dense mixed graphs of undirected degree $q$, directed degree $\frac{q-1}{2}$ and order $2q^2$, for $q$ being an odd prime power. Since the  Moore bound  for a mixed graph with these parameters is equal to $\frac{9q^2-4q+3}{4}$ the defect of these mixed graphs is $({\frac{q-2}{2}})^2-\frac{1}{4}$.

In particular we obtain a known mixed Moore graph of order $18$, undirected degree $3$ and directed degree $1$  called Bos\'ak's graph and a new mixed graph of order $50$, undirected degree $5$ and directed degree $2$, which is proved to be optimal.
\end{abstract}


\section {Introduction}

In this paper we consider graphs which are finite and mixed, i.e., they may contain (directed edges) arcs as well as undirected ones. The mixed  graphs are also called {\it{partially directed graphs}}.  Bos\'ak  \cite{B79} investigated those mixed graphs with given degree and given diameter having maximum number of vertices which are called  Mixed Moore graphs. In some sense, Bos\'ak generalized the concepts of Moore graph and Moore digraph by allowing the existence of both edges and arcs simultaneously. These  graphs and digraphs have been very used to model different kind of networks (such as a telecommunications, multiprocessor, or local area network, to name just a few). In many real-world networks a mixture of both unidirectional and bidirectional connections may exist (e.g. the World Wide Web network, where pages are nodes and hyperlinks describe the connections). For such networks, mixed graphs provide a perfect modeling framework \cite{MS13}.

\noindent Undirected graphs (mixed Moore graphs admitting only edges) with maximum degree $d$ and diameter $k$ are graphs of order $M_{d,k}=1+d+d(d-1)+\cdots+d(d-1)^{k-1}$ (undirected Moore bound). There are no Moore graphs of degree $d\geq 3$ and diameter $k \geq 3$, see \cite{BI73, D73, HS60}.
For $k=1$ and $d\geq 1$ complete graphs $K_{d+1}$ are the only Moore graphs. For $k\geq 3$ and $d=2$ the cycles $C_{2k+1}$ are the only Moore graphs. For $k=2$, apart from $C_5$ $(d=2)$, Moore graphs exist only when $d=3$  (Petersen graph), $d=7$ (the Hoffman-Singleton graph) and possibly $d=57$.
\noindent For more details and results concerning Moore graphs see the survey by Miller and \v{S}ir\'a\v{n} \cite{MS13}.

\noindent Directed Moore graphs (mixed Moore graphs admitting only arcs) with maximum out-degree $d$ and diameter $k$ are digraphs of order $M^*_{d,k}=1+d+d^2+\cdots+d^k$ (directed Moore bound). In \cite{BT80, PZ74} it was proved that the Moore digraphs do not exist for $d>1$ and $k>1$. The unique Moore digraphs are directed cycles of length $k+1$, denoted by  $\vec{C}_{k+1}$ and complete graphs on $d+1$ vertices.

\vskip3mm

 Directed and undirected Moore graphs are special cases of mixed Moore graphs where the graphs admit only arcs or only edges, respectively \cite{MM08}. A mixed graph is said to be {\it{proper}} if it contains at least one arc and at least one edge. In particular,
a {\it{mixed regular graph}} is a simple and finite graph $G$ where each vertex $v$ of $G$ is incident with $z$ arcs from it and $r$ edges; $z$ is the directed degree and $r$ is the undirected degree of $v$ and we set $d=r+z$, $d$ being the degree of $v$. If $G$ has diameter equal to $k$, we say that $G$ is a $(z,r;k )$-mixed graph of directed degree $z$, undirected degree $r$ and diameter $k$.

\noindent Let $M_{z,r,k}$ denote the upper bound on the order of a $(z,r;k)$-mixed graph. A mixed graph that attains this bound is called a $(z,r;k)$-mixed Moore Graph of diameter $k$. 
Note that $M_{z,r,k}=M_{d,k}$ when $z=0$ and $M_{z,r,k}=M^*_{d,k}$ when $r=0\  (d=z+r)$.

The following conjecture was proposed by Bos\'ak in 1979 (see \cite{B79}) and proved in 2007 by Nguyen, Miller and Gimbert (see \cite{MMG07}).


\begin{conjecture}\label{BC}\cite{B79}.
\noindent Let  $d \geq 1, k \geq 3$ be two integers. A finite graph $G$ is a mixed Moore graph of degree $d$ and diameter $k$ if and only if either $d=1$ and $G$ is $\vec{C}_{k+1}$, or $d=2$ and $G$ is $C_{k+1}$.
\end{conjecture}

\noindent Hence, mixed Moore graphs  different from cycles have diameter $k=2$ and their order attains the following upper bound:
\begin{equation}\label{upper} M_{z,r,2}= 1+r+z+r(r-1+z)+z(r+z)= (r+z)^2+z+1.\end{equation}

\noindent  Bos\'ak in \cite{B79} gives for $k=2$ divisibility conditions for the existence of Mixed Moore graphs related with the distribution of undirected and directed edges. He proved that the two parameters $z$ and $r$ must satisfy a tight arithmetic condition. Accordingly, apart from the trivial cases when  $r = 0$ and  $z = 1$   (graph $\vec{C}_3$),  $r = 2$ and $z = 0$   (graph $C_5$), there must exists a positive odd integer $c$ such that
\begin{equation}\label{c} c | (4z - 3)(4z + 5) \mbox{ and } r = \frac{1}{4} (c^2 + 3). \end{equation}
In the same paper Bos\'ak provides constructions of some of these mixed Moore graphs with all of them being, except the Bos\'ak graph of order $n=18$, isomorphic to Kautz digraphs $K_a(d,2)$ with all digons (cycles of length 2) considered as undirected edges (see \cite{K68, K69}).

In 2007 Nguyen, Miller and Gimbert (see \cite{MMG07}) proved that all mixed Moore graphs of diameter 2 known at that time were unique. However, this is not generally true since Jorgensen recently found  (see \cite{J13}) two non-isomorphic mixed Moore graphs of diameter 2,  out-degree 7, undirected degree 3 and order 108.



\noindent Table \ref{all} depicts all values for $n\leq 200 $ with the corresponding feasible values of $z$ and $r$, such that a mixed Moore graph either exists or is not known to exist. There are still many values of $r$ and $z$ for which the existence of a mixed Moore graph of diameter 2 has not been settled.

In both the undirected and directed graphs most of the work carried out on this topic has focused on constructing (di)-graphs of diameter $k\geq 2$, degree $d\geq 3$ and a number or vertices as close as possible to the respective Moore bounds. Based on this,  we say that a mixed Moore graph is {\it{dense}} if its order is close to the Moore bound and it is  {\it{optimal}}, if there does not exist a mixed graph with the same parameters and bigger order.

In this paper we give a construction of mixed graphs of diameter $2$, undirected degree $q+2t$, directed degree $(q-1)/2-2t$ and $2q^2$ vertices for $q$ being an odd prime power and either $ t \in \{0,\ldots,\frac{q-1}{4}\}$, if $q\equiv 1$ $(mod ~4)$, or $ t \in \{0,\dots, \frac{q-3}{4}\}$, if $q\equiv 3$ $(mod ~4)$. In particular, when $t=\frac{q-1}{4}$ and $q\equiv 1$ $(mod 4)$, ~we reobtain the undirected graphs constructed
by McKay, Miller and \v{S}ir\'a\v{n} in 1998 which are currently the largest known with diameter $2$. For $t=0$ the constructed family of  $(\frac{q-1}{2},q)$-mixed graphs of diameter 2 are dense, since the  Moore bound  for a mixed Moore graph with these parameters is equal to $\frac{9q^2-4q+3}{4}$, so it follows that the defect of these mixed graphs is $({\frac{q-2}{2}})^2-\frac{1}{4}$.

Furthermore, our construction provides for $q=3$, a known $(1,3;2)$-mixed Moore graph of order $18$ called Bos\'ak's graph and a new $(2,3;2)$-mixed graph of order $50$, which is proved to be optimal.  For the rest of the values of $q$ and $t$ our construction produces good lower bounds for the degree/diameter problem.

{\small
\begin{table}[ht!]
\begin{center}
\begin{tabular}{|c|c|c|c|c|c|c|c||c|c|c| }\hline
 n & d & z & r  & existence & uniqueness \\ \hline\hline
3    & 1   & 1  & 0 & $Z_3$   &       YES\\ \hline
5    & 2   & 0  & 2 & $C_5$   &       YES \\ \hline
6    & 2   & 1  & 1 & $Ka(2,2)$       &YES \\ \hline
10   & 3   & 0  & 3 & Petersen graph  &  YES  \\ \hline
12   & 3   & 2  & 1 & $Ka(3,2)$       &  YES \\ \hline
18   & 4   & 1  & 3 & Bos\'ak graph   & YES \\ \hline
20   & 4   & 3  & 1 & $Ka(4,2)$       & YES \\ \hline
30   & 5   & 4  & 1 & $Ka(5,2)$       & YES   \\ \hline
40   & 6   & 3  & 3 & unknown         & unknown   \\ \hline
42   & 6   & 5  & 1 & $Ka(6,2)$       & YES     \\ \hline
50   & 7   & 0  & 7 & Hoffman-Singleton graph  & YES   \\ \hline
54   & 7   & 4  & 3 & unknown   & unknown  \\ \hline
56   & 7   & 6  & 1 & $Ka(7,2)$ & YES    \\ \hline
72   & 8   & 7  & 1 & $Ka(8,2)$ & YES      \\ \hline
84   & 9   & 2  & 7 & unknown    & unknown   \\ \hline
88   & 9   & 6  & 3 & unknown & unknown   \\ \hline
90   & 9   & 8  & 1 & $Ka(9,2)$    & YES      \\ \hline
108  & 10  & 7  & 3 & Jorgensen graph    & NO      \\ \hline
110  & 10  & 9  & 1 & $Ka(10,2)$    & YES      \\ \hline
132  & 11  & 10 & 1 & $Ka(11,2)$    & YES      \\ \hline
150  & 12  & 5  & 7 & unknown   & unknown      \\ \hline
156  & 12  & 11 & 1 & $Ka(12,2)$    & YES      \\ \hline
180  & 13  & 10 & 3 & unknown    & unknown     \\ \hline
182  & 13  & 12 & 1 & $Ka(13,2)$    & YES      \\ \hline

\end{tabular}
\caption{\label{all} Feasible values of the parameters for proper mixed Moore graphs of order up to 200.}
\end{center}
\end{table}
}





\section {Notation and terminology}

\noindent Let $G$ be a mixed graph with vertex set $V(G)$, edge set $E(G)$ and arc set $A(G)$. The distance from a vertex $u$ to a vertex $v$  is the length of a shortest path from $u$ to $v$. The distance from $u$ to $v$ is infinite, if $v$ is not reachable from $u$. Note that in a directed graph the distance from vertex $u$  to vertex $v$ can differ from the distance  from $v$ to $u$. The maximum value $k$ of the distance over all pairs of vertices of $G$ is the diameter of the graph.



In our constructions we use the incidence graph of a partial plane. A {\it partial plane}    is defined   as
   two finite sets  $\mathcal{P}$ and $\mathcal{L}$ called points and lines respectively, where $\mathcal{L}$  consists of subsets of $\mathcal{P}$, such that
 any line is incident with at least two points, and two points are
incident with at most one line.
The  {\it incidence graph} of a partial plane is a bipartite graph with partite sets $\mathcal{P}$
and     $\mathcal{L}$ and a point of $\mathcal{P}$ is adjacent to a line of $\mathcal{L}$ if they are incident. 
  Observe that the incidence graph of a partial plane
is clearly a bipartite graph with even girth $g\ge 6$.


The following definition of a biaffine plane was given in \cite{B67}:
\begin{definition} \label{def}Let $\mathbb{F}_q$ be the finite field of order $q$.
 \begin{itemize}

 \item[(i)]Let $\mathcal{L}=   \mathbb{F}_q  \times  \mathbb{F}_q  $ and  $\mathcal{P}=   \mathbb{F}_q   \times   \mathbb{F}_q  $ denoting the elements of $\mathcal{L}$ and $\mathcal{P}$ using ``brackets" and ``parenthesis", respectively.

\noindent The following set  of $q^2 $ lines define a  biaffine plane:
$$  [m,b] =\{(x, mx+b)  : x\in \mathbb{F}_q\}   \mbox{ for all } m,b\in \mathbb{F}_q.
 $$
\item[(ii)] The incidence graph of the biaffine plane is a bipartite graph $B_q=(\mathcal{P},\mathcal{L})$ which
is $q$-regular, has order $2q^2$,  diameter $4$ and girth $6$,  if $q\ge 3$; and girth $8$, if $q=2$.
\end{itemize}
\end{definition}
\begin{remark} \label{rem}
The diameter of $B_q$ is 4 and the vertices mutually at distance 4 are the vertices of the sets   $L_m=\{[m, b]: b\in \mathbb{F}_q\}$,  and $P_x=\{( x, y): y\in \mathbb{F}_q\}$  for all $x,m\in \mathbb{F}_q$.
\end{remark}

\section{An infinite family of dense mixed graphs}
\subsection{Basic Construction}

Let $q$ be an odd prime power. Let $\mathbb{F}_q$ be the finite field of order $q$ and let $M\subseteq \mathbb{F}_q-0$, $|M |  =(q-1)/2$, be such that for all $u,v\in M$, $u+v\ne 0$. Therefore  $|-M |  =(q-1)/2$, $M\cap M=\emptyset$  and $M\cup (-M)=\mathbb{F}_q-0$.
Let $T\subseteq M$ with $|T|=2t$ be such that  $ t \in \{0,\ldots, \frac{q-1}{4}\}$ if $q\equiv 1$ $(mod ~ 4)$, or $ t \in \{0,\ldots, \frac{q-3}{4}\}$ if $q\equiv 3$ $(mod ~4)$.

\noindent Let $-T=\{t'\in -M : t+t'=0, \mbox{ for } t\in T\}$ and consider the following two sets:
 $$T_1=\{t_i,t_i', i=1,\ldots, t:t_i+t'_i=0\}\subseteq T\cup(-T) \mbox{ and }$$
  $$T_2=\{t_i, t'_i, i=t+1,\ldots, 2t:t_i+t'_i=0\}\subseteq T\cup(-T).$$
 Let $S=M-T$ and $(-S)=(-M)-(-T)$.\\

\begin{construction}\label{oddq}Let $q\ge 3$ be an  odd prime power. Using the aforementioned  defined sets we construct a mixed graph $G_{q,t}$ for any $ t \in \{0,\ldots, \frac{q-1}{4}\}$, if $q\equiv 1$ $(mod~4)$ or $ t \in \{0,\ldots, \frac{q-3}{4}\}$ if $q\equiv 3$ $(mod~ 4)$ as follows: Let $V(G_{q,t})=V(B_q)$, 

\noindent $E(G_{q,t})=E(B_q)\cup \{( [m,b], [m,b+i]): i\in T_1 \}\cup \{ ((x,y), (x,y+j)): j\in T_2 \}$ and 

\noindent $A(G_{q,t})=\{( [m,b], [m,b+i]): i\in S \}\cup \{ ((x,y), (x,y+j)): j\in -S \}$.
 \end{construction}

 In Figure \ref{b11} we exhibit, for $q=7$, the induced mixed subgraphs of $G_{7,t}$ for points and lines when $t=0$ (on the left) and $t=1$ (on the right).

\begin{figure}[h]
\begin{center}
\includegraphics[height =10cm]{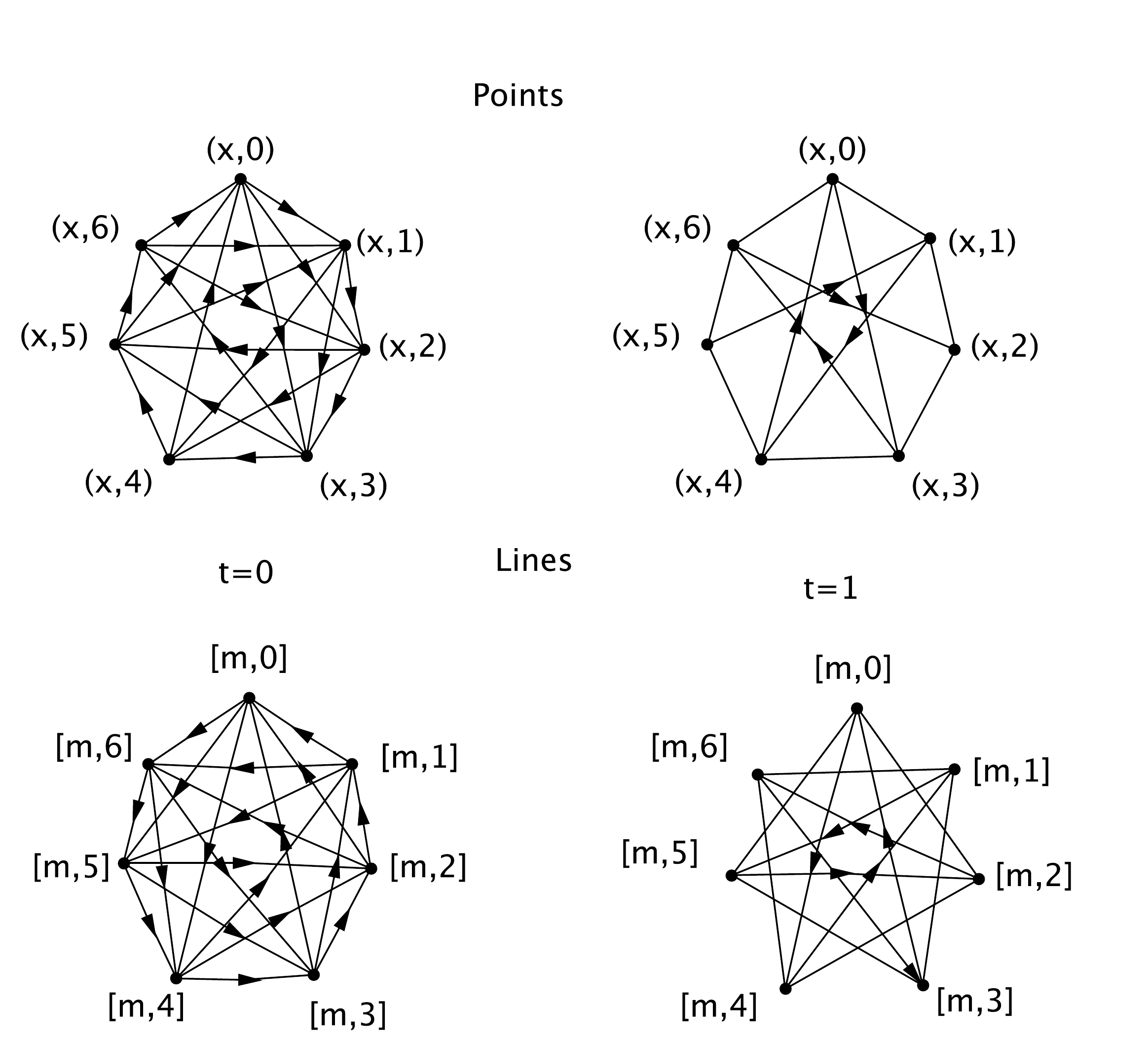}
\end{center}
\caption{The induced mixed subgraphs of $G_{7,t}$ for points and lines for $t=0$ (on the left) and $t=1$ (on the right). }\label{b11}
\end{figure}

 \begin{theorem}\label{cons}
 The mixed graph $G_{q,t}$ defined in Construction \ref{oddq} is a mixed graph of diameter 2 with parameters  $r=q+2t$, $z=(q-1)/2-2t$ and $2q^2$ vertices.
 \end{theorem}
\begin{proof} It is immediate that the parameters of $G_{q,t}$ are   $r=q+t$, $z=(q-1)/2-2t$ since $B_q$ is a $q$-regular graph and we have added to each vertex $(q-1)/2-2t$ outing arcs.
Let us show that the diameter is 2. First, observe that  given any two vertices $[m,b]$ and  $[m,b']$, the set of their adjacent vertices in $B_q$ are by definition $(x,mx+b)$ and $(x,mx+b')$ for all $x\in \mathbb{F}_q$, respectively, and we have four possibilities in $G_{q,t}$:

\begin{itemize}
\item [(i)] If $b'-b\in T_1$ then $[m,b]$ and  $[m,b']$ are adjacent and $(x,mx+b)$ and $(x,mx+b')$ are not adjacent.
\item [(ii)]  If $b'-b\in T_2$ then $[m,b]$ and  $[m,b']$ are not adjacent and $(x,mx+b)$ and $(x,mx+b')$ are adjacent.
\item [(iii)] If $b'-b\in S$ then there exists an arc from $[m,b]$ to $[m,b']$ and an arc to $(x,mx+b)$ from $(x,mx+b')$.
\item [(iv)]  If $b'-b\in (-S)$ then there exists an arc to $[m,b]$ from  $[m,b']$ and an arc from $(x,mx+b)$ to $(x,mx+b')$.
\end{itemize}

Let us check the distance $d_{G_{q,t}}([m,b],(x,y))$  for any pair of vertices $[m,b],(x,y)$ assuming that they are not adjacent because otherwise we are done.
Suppose that $[m,b']$ is adjacent to $(x,y)$ in $B_q$; if $b'-b\in T_1\cup S$, then by (i) and (ii), $[m,b],[m,b'] ,(x,y)$ is a path of length two in ${G_{q,t}}$ and $d_{G_{q,t}}([m,b],(x,y))= 2$. \\
Now, if $b'-b\in T_2\cup (-S)$ since $[m,b']$ is adjacent to $(x,y)$ in $B_q$ we have  $(x,y)=(x,mx+b')$ and also $[m,b]$ is adjacent to $(x,mx+b)$ in $B_q$, then by (ii) and (iv) it follows that $[m,b],(x,mx+b),(x,y)$ is a path of length two in ${G_{q,t}}$ and $d_{G_{q,t}}([m,b],(x,y))= 2$.

 Consequently we conclude that $d_{G_{q,t}}([m,b],(x,y))\leq 2$ for any pair of vertices $\{[m,b],(x,y)\}$.

 Let us check the distance $d_{G_{q,t}}([m,b],[m,b'])$. If $b'-b\in T_1\cup S$ then $d_{G_{q,t}}([m,b],[m,b'])=1$. Therefore we assume that  $s=b'-b\in T_2\cup (-S)$, that is, either $[m,b]$ and $[m,b']$ are not adjacent or there is an arc to $[m,b]$ from  $[m,b']$. Observe that the set $A_1=\{[m,b+s]: s\in T_1\cup (S)\}$ has  $(q-1)/2$ vertices. Moreover the set  $A_2=\{[m,b'-s]: s\in T_2\cup(-S)\}$ has  $(q-1)/2$ vertices. If $A_1\cap A_2=\emptyset$ then the set $V_m=\{[m, b]: b\in \mathbb{F}_q\}=A_1\cup A_2\cup \{[m,b],[m,b']\}$,  implying that $|V_m|=q+1$,  which is a contradiction because $|V_m|=q$. Thus $A_1\cap A_2\not=\emptyset$ yielding that there exists some $s\in S$ such that $[m,b],[m,b'-s],[m,b']$ is a path of length two in ${G_{q,t}}$. Thus in either case $d_{G_{q,t}}([m,b],[m,b'])\le 2$.

Analogously, it is proved that   $d_{G_{q,t}}((x,y),(x,y'))\le 2$. Hence we can conclude that the diameter of $G_{q,t}$ is 2.
\end{proof}

From Theorem \ref{cons}  the following corollaries are immediate.

\begin{corollary}
\label{tzero}
For $q$ being an odd prime power and $t=0$ the graph, called for simplicity $G_{q}$, given in Theorem \ref{cons} is a mixed graph of diameter 2 with parameters  $r=q$, $z=(q-1)/2$ and $2q^2$ vertices.
\end{corollary}

In this case $S=M$ and $(-S)=(-M)$; and $E(G_q)=E(B_q)$, $A(G_q)=\{( [m,b], [m,b+i]): i\in M \}\cup \{ ((x,y), (x,y+j)): j\in -M \}$.
Figure \ref{G5} depicts $G_5$.

\begin{figure}[h]
\begin{center}
\includegraphics[height =9cm]{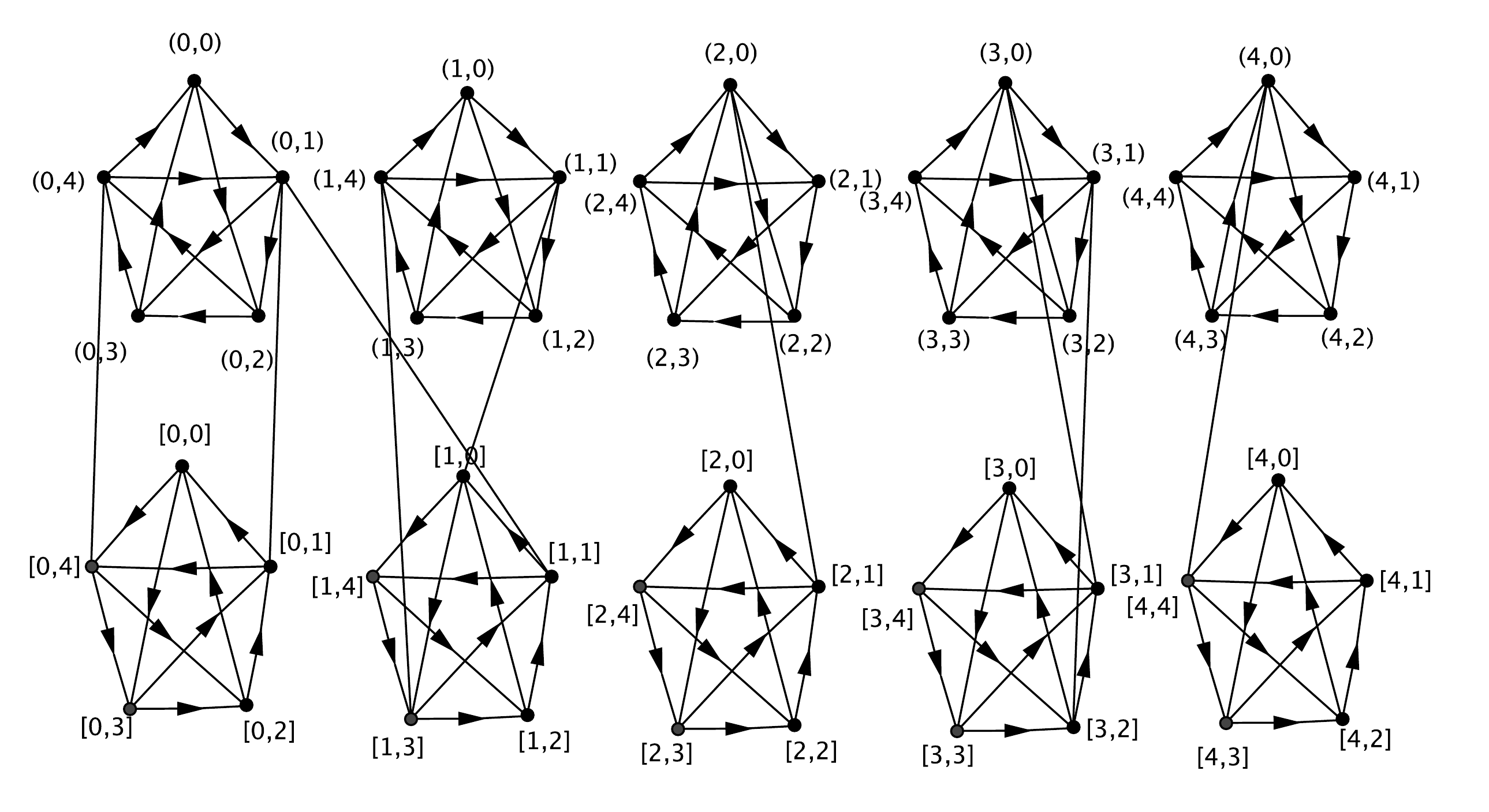}
\end{center}
\caption{graph $G_5$ }\label{G5}
\end{figure}

\begin{corollary}
\label{zzero}
 Let $q\equiv 1\ \ (mod ~4)$ be an odd prime power and $t=\frac{q-1}{4}$. Then the mixed graph $G_{q,\frac{q-1}{4}}$ given in Theorem \ref{cons}  has diameter 2 and parameters  $r=\frac{3q-1}{2}$, $z=0$ and $2q^2$ vertices.
\end{corollary}

The graphs $G_{q,\frac{q-1}{4}}$  have already been given by various authors using different techniques. The first construction was made by McKay, Miller and \v{S}ir\'a\v{n} in \cite{MS13} using lifts and voltage graph. Afterwards, it was constructed by Hafner \cite{H04} and by Araujo, Noy and Serra in \cite{ANS06} using geometrical techniques. It is important to take into consideration that this family of graphs $G_{q,\frac{q-1}{4}}$ is the largest known until now related to the Moore bound for given parameters and diameter $2$ (see \cite{MS13} or \cite{table}). Moreover, the graphs constructed by these authors are vertex transitive, in particular the graph $G_{q,\frac{q-1}{4}}$ for any $q\equiv 1 (mod~ 4)$ odd prime power is vertex transitive. Furthermore, in \cite{BMSZ13} there is a complete discussion devoted to  symmetries and automorphism groups of these graph and other related issues.

\section {Vertex transitivity of $G_q$}
In this section we prove that graph $G_q:=G_{q,0}$ of Corollary \ref{tzero} is vertex transitive. In addition, we provide a short remark explaining why $G_{q,t}$ is not vertex transitive for $1\leq t \leq \frac{q-1}{4}-1$ with $q\equiv1$ $(mod ~4)$ and  for $1\leq t \leq \frac{q-3}{4}$ with $q\equiv 3$ $(mod ~4)$. Consequently, $G_{q,t}$ is vertex transitive only if $t=0$ and $t=\frac{q-1}{4}$ and $q\equiv 1$  $(mod ~4)$ is an odd prime power.

\begin{theorem}\label{Gqtransitive}
Let $q$ be an odd prime power then $ G_q$ is a vertex transitive mixed graph.
 \end{theorem}
\begin{proof}
Note that the function $\theta$ that interchanges the line $[m,b]$ at point $(-m,-b)$, and  similarly point $(x,y)$ at line $[-x,-y]$, is an automorphism of $B_q$ which exchanges stable sets of the graph.

We will prove that $\theta$ is  an automorphism of $G_q$.  Let $v\in V(G_q)$, we will prove that $N^+(\theta(v))=\theta(N^+(v))$ where $N^+(v)$ denotes the exneighborhood of $v$ (the set of vertices that receives an arc from $v$).
Let $[m,b]\in \mathcal{L}$ be a line of $G_q$. Since $N^+([m,b])=\{[m,b+i]: i\in M\}$ it follows that
$$\theta(N^+([m,b]))=\{(-m,-b-i): i\in M\}=\{(-m,-b+j): j\in -M\}=N^+(\theta[m,b]).$$
A similar argument is used to prove that $\theta(N^+((x,y)))= N^+(\theta(x,y))$ for $(x,y)\in \mathcal{P}$ a point of $G_q$.

Let us define the function $$\Psi_{(a,t)}:V(G_q)\longrightarrow V(G_q)\ \mbox{such that}  $$
$$\Psi_{a,t}([m,b])=[-m,b+am+t]$$
$$\Psi_{a,t}((x,y))=(-x+a,y+t)$$
 Next, we prove that $\Psi_{a,t}$ is an automorphism of $G_q$.
First, we prove that:
$$N(\Psi_{a,t}[m,b])\cup N^+(\Psi_{a,t}[m,b]) =\Psi_{a,t}(N([m,b])\cup N^+([m,b])).$$
Note that $$N([m,b])\cup N^+([m,b])=\{(x,mx+b) : x\in \mathbb{F}_q\} \cup \{[m,b+i]:\  i\in M\}.$$
Then
$$\begin{array}{lll} N(\Psi_{a,t}[m,b])\cup N^+(\Psi_{a,t}[m,b])&=&
N([-m,b+am+t])\cup N^+([-m,b+am+t])\\ &=& \{(x',-mx'+b+am+t): x'\in \mathbb{F}_q\} \\&& \cup ~~ \{[-m,b+am+t+i]: i\in M\}.\end{array}$$
Note that if $x'=-x+a$, then $-m(-x+a)+b+am+t=mx+b+t$ and
{\small
$$\begin{array}{ll} N(\Psi_{a,t}[m,b])\cup N^+(\Psi_{a,t}[m,b])&=\{(-x+a,mx+b+t) : x\in \mathbb{F}_q\} \cup \{[-m,b+am+t+i]: i\in M\}\\ &= \Psi_{a,t}(N([m,b])\cup(N^+([m,b])).\end{array}$$
}
Let us  also prove that:
$$N(\Psi_{a,t}(x,y))\cup N^+(\Psi_{a,t}(x,y)) =\Psi_{a,t}(N(x,y)\cup(N^+(x,y)).$$
We have
 $N(x,y)\cup N^+(x,y)=\{[m,y-my] : m\in \mathbb{F}_q\} \cup \{(x,y+i): i\in -M\}$. Then
$$N(\Psi_{a,t}(x,y))\cup N^+(\Psi_{a,t}(x,y))=N(-x+a,y+t) \cup N^+(-x+a,y+t)$$
$$=\{[m',y+t-m'(-x+a)] : m'\in \mathbb{F}_q\} \cup \{(-x+a,y+t+i): i\in -M\}.$$
Note that if $m'=-m$, then
$$N(\Psi_{a,t}(x,y)\cup N^+(\Psi_{a,t}(x,y))=\{[-m,y+t-mx+ma)] : m\in \mathbb{F}_q\} \cup \{(-x+a,y+t+i): i\in -M\}$$
$$=\{[-m,b+am+t] : m \in \mathbb{F}_q\} \cup \{(-x+a,y+t+i): i\in -M\}=\Psi_{a,t}(N(x,y)\cup N^+(x,y)).$$

Now, if we compose $\theta$ with $\Psi_{a,t}$ we obtain an automorphism which exchanges any pair of elements of $G_q$, consequently, $G_q$ is vertex-transitive.  Indeed, let $[m,t]$ and $(x,y)$ be two vertices of $G_q$. We have  $\Psi_{x-m,y+b}\circ \theta ([m,t])=(x,y)$. Thus,   there  always  exists a composition that sends any element to another.
\end{proof}

\begin{remark}
The graph $G_{q,t}$ for any $ t \in \{1,\ldots, \frac{q-1}{4}-1\}$, if $q\equiv 1$ $(mod ~4)$ or $ t \in \{1,\ldots, \frac{q-3}{4}\}$ if $q\equiv 3$ $(mod ~4)$ is not vertex transitive.
\end{remark}
If $G_{q,t}$ is vertex transitive, then the automorphism (or composition of automorphisms) should exchange stable sets (points and lines). Moreover, it should send sets of vertices $L_m$ for $m\in\mathbb{F}_q$ to sets of vertices $P_x$ for $x\in\mathbb{F}_q$. However, it is not difficult to observe that by definition of $G_{q,t}$, the automorphisms  preserving adjacencies between any pair of sets $L_m$ and $P_x$ do not preserve arcs.


\section{Optimal Mixed  Graphs}

In this section we provide some results related to the family constructed in Corollary \ref{tzero}. The following lemma is used in what follows.

\begin{lemma}\label{parity} Let $r\geq 1$ be an odd integer. Then there is not a $(z,r;k)$-mixed graph of odd order.
\end{lemma}
\begin{proof}
 Suppose that there is a $(z,r;k)$-mixed graph of odd order with $z,r\geq 1$ and $r$ odd. Deleting the directions of the arcs we obtain a regular graph of odd degree $2z+r$ and odd order, which is a contradiction.
\end{proof}

\begin{remark}
We construct a family of dense $(\frac{q-1}{2},q)$-mixed graphs of diameter 2. Since the  Moore bound  for a mixed Moore graph with these parameters is equal to $\frac{9q^2-4q+3}{4}$ the defect of these mixed graphs is $({\frac{q-2}{2}})^2-\frac{1}{4}$.
\end{remark}

\begin{theorem} For $q=3$, $G_3$ is a mixed Moore graph and for $q=5$, $G_5$ is an optimal  mixed graph.
\end{theorem}

\begin{proof} For $q=3$ it turns out that $G_3$ has 18 vertices and parameters $r=3$ and $z=1$. Since Bosak's graph is unique, see \cite{MMG07}, we obtain that $G_3$ given in Construction \ref{oddq} is isomorphic to Bosak's graph.

For $q=5$ it turns out that $G_5$ has 50 vertices and parameters $r=5$ and $z=2$. By (\ref{upper}) the upper bound on the number of vertices for this particular case is $52$. Let us show that a mixed Moore graph with $52$ vertices  and parameters $r=5$ and $z=2$ cannot exist.   Otherwise, by (\ref{c}) an odd integer $c$ dividing $(4z-3)(4z+5)=65$ exists, such that $r=5=\frac{1}{4}(c^2+3)$.
But then  $c=\sqrt{17}$ which implies that $c$ is not an integer.   Therefore   the upper bound on the number of vertices  must be at most 51. 
However, from Lemma \ref{parity} it follows that   there is no graph of order 51, and we conclude that the upper bound is 50, yielding that $G_5$  is an optimal  mixed graph.
\end{proof}

 In Figure \ref{G5} we show the optimal $(2,5;2)$-mixed graph.

To conclude the paper we pose a problem which could be studied in the future.

{\bf Problem:}

For $q=7$, $r=7$ and $z=3$  the Moore Bound is equal to $104$. Our construction provides a $(3,7;2)$-mixed graph on $98$ vertices.
By Bosak's condition we know that the Moore bound is not attainable, thus we need to know whether there is or there is not a $(3,7;2)$-mixed graph on either $100$ or $102$ vertices.\\
Note that for these parameters, by Lemma \ref{parity}, graphs with odd order are not possible.





\subsection*{Acknowledgment}
{ Research   supported by the Ministerio de Educaci\'on y Ciencia,
Spain, and the European Regional Development Fund (ERDF) under
project MTM2014-60127-P;  CONACyT-M\'exico under projects 57371,166306, and PAPIIT-M\'exico under project 104609-3}. The last author acknowledgment support by the Slovak Research Grant: the Vega Research Grant 1/0065/13.


\end{document}